\documentclass[11pt]{amsart}
\usepackage{geometry}                
\usepackage{graphicx}
\usepackage{amssymb, amsmath, amsthm, amssymb, amsfonts}
\usepackage{epstopdf}
\newtheorem{Proposition}{Proposition}[section]
\newtheorem{theorem}{Theorem}

\newtheorem{lemma}[Proposition]{Lemma}
\newtheorem{prop}[Proposition]{Proposition}
\newtheorem{cor}[Proposition]{Corollary}

\newcommand{\sss}{\mathfrak{S}}

\newcommand{\isomorph}{\cong}

\newcommand{\curlyleq}{\preccurlyeq}
\newcommand{\curlygeq}{\succcurlyeq}

\title{Powers of Coxeter Elements in Infinite Groups are Reduced}
\author{David E Speyer}

\begin{document}

\begin{abstract}
 Let $W$ be an infinite irreducible Coxeter group with $(s_1, \ldots, s_n)$ the simple generators. We give a simple proof that the word $s_1 s_2 \cdots s_n s_1 s_2 \cdots s_n \cdots s_1 s_2 \cdots s_n$ is reduced for any number of repetitions of $s_1 s_2 \cdots s_n$. This result was proved for simply-laced, crystallographic groups by Kleiner and Pelley using methods from the theory of quiver representations. Our proof only using basic facts about Coxeter groups and the geometry of root systems.
\end{abstract}

\maketitle

Let $W$ be a Coxeter group with $S$ the generating set of reflections. An element $c \in W$ of the form $s_1 \cdots s_n$, with $s_1$, \dots, $s_n$ some ordering of the elements of $S$, is called a Coxeter element. It is a result of Howlett \cite{Howlett} that, if $W$ is infinite, then any Coxeter element has infinite order.
 In \cite{KMOTU}, it is shown that, in each classical affine group, there is an ordering $(s_1, \ldots, s_n)$ of the simple generators such that the word $s_1 \cdots s_n s_1 \cdots s_n \cdots s_1 \cdots s_n$ is reduced for any number of repetitions of $s_1 \cdots s_n$.\footnote{More specifically, the authors of \cite{KMOTU} define four properties of a sequence $r_1$, $r_2$, \dots of simple reflections; property (IV) is that the word $r_1 r_2 \cdots r_N$ is reduced for any $N$. For each classical affine type, they exhibit a sequence of reflections which satisfies their properties and the ordering is periodic in each case.} Fomin and Zelevinsky \cite[Corollary 9.6]{FZ} proved a version of this result for Coxeter groups with bipartite diagrams; they show that, if $S = I \sqcup J$ is a partition of $S$ into two sets so that all the elements in each set commute, and if $W$ is irreducible and infinite then the word $\prod_{i \in I} s_i \prod_{j \in J} s_j \prod_{i \in I} s_i \prod_{j \in J} s_j \cdots \prod_{i \in I} s_i \prod_{j \in J} s_j$ is reduced for any number of repetitions of $\prod_{i \in I} s_i \prod_{j \in J} s_j$. Recently, Kleiner and Pelley \cite{KP}, relying heavily on results of Kleiner and Tyler \cite{KT}, have used methods from quiver representation theory to show that, if $W$ is a simply-laced, crystallographic Coxeter group which is irreducible and infinite then the word $s_1 \cdots s_n s_1 \cdots s_n \cdots s_1 \cdots s_n$ is reduced for any number of repetitions of $s_1 \cdots s_n$. It is trivial to extend this result to the case where $W \isomorph W_1 \times W_2 \times \cdots W_r$, with each $W_i$ a Coxeter group meeting the above conditions.

The aim of this note is to reprove Kleiner and Pelley result using only the theory of Coxeter groups and the geometry of root systems. Our proof is inspired by that of Kleiner and Pelley, but we strip out the quiver theory and simplify several arguments. In the process, we strip out the assumptions that $W$ is crystallographic and simply-laced. To repeat, our result is:

\begin{theorem} \label{Result1}
 Let $W$ be an infinite, irreducible Coxeter group and let $(s_1, \cdots s_n)$ be any ordering of the simple generators. Then the word $s_1 \cdots s_n s_1 \cdots s_n \cdots s_1 \cdots s_n$ is reduced for any number of repetitions of $s_1 \cdots s_n$.
\end{theorem}

 Our primary tool is the introduction of a skew-symmetric form $\omega_c$ on the root space. In a forthcoming paper, Nathan Reading and I will use this form to generalize Reading's results on sortable elements to infinite Coxeter groups.

\section{Conventions regarding Coxeter Groups}

Let $W$ be a Coxeter group of rank $n$. That means that $W$ is generated by $s_1$, \dots, $s_n$, subject to the relations $s_i^2=1$ and $(s_i s_j)^{m_{ij}}=1$ for $i \neq j$ where $2 \leq m_{ij}=m_{ji} \leq \infty$. The Dynkin diagram of $W$ is the graph $\Gamma$ whose vertices are labeled $1$, \dots, $n$ and where there is an edge between $i$ and $j$ if $m_{ij} \neq 2$. The group $W$ is called irreducible if $\Gamma$ is connected. An element of the form $s_{x_1} \cdots s_{x_n}$ of $W$, for some permutation $x_1 \cdots x_n$ of $\{ 1,\ldots,n \}$, is called a Coxeter element. Given such a permutation, direct $\Gamma$ such that $i \to j$ if $x_i > x_j$. Two permutations yield the same Coxeter element if and only if they give rise to the same orientation of $\Gamma$, so in this way we get a bijection between Coxeter elements and acyclic orientations of $\Gamma$.

Let $V$ be the $n$-dimensional real vector space with basis $\alpha_1$, \dots $\alpha_n$ and equip $V$ with the symmetric bilinear form $B$ such that $B(\alpha_i, \alpha_i)=2$ and $B(\alpha_i, \alpha_j)=-2 \cos (\pi/m_{ij})$ for $i \neq j$. Then $W$ acts on $V$ by $s_i : v \mapsto v-B(v,\alpha_i) \alpha_i$ and this action preserves the bilinear form $B$. The elements of $V$ of the form $w \alpha_i$ are called roots.\footnote{Those interested in Kac-Moody algebras and quiver theory would call these the real roots; those from a Coxeter theoretic background would simply call them roots. We follow the latter convention.} Every root is either in the positive real span of the $\alpha_i$, in which case it is called a positive root, or in the positive real span of the $- \alpha_i$, in which case it is called a negative root. The positive roots are in bijection with the reflections, via $w \alpha_s \leftrightarrow w s w^{-1}$. We write $\alpha_t$ for the positive root associated to the reflection $t$. We have $w \alpha_t=\pm \alpha_{wtw^{-1}}$.

For any $w \in W$, the set of inversions of $w$ is defined to be the set of reflections $t$ such that $w^{-1} \alpha_t$ is a negative root. If we write $w$ as $s_{x_1} \cdots s_{x_N}$, then the inversions of $w$ are the reflections that occur an odd number of times in the sequence $s_{x_1}$, $s_{x_1} s_{x_2} s_{x_1}$, $s_{x_1} s_{x_2} s_{x_3} s_{x_2} s_{x_1}$, \dots , $s_{x_1} s_{x_2} \cdots s_{x_n} \cdots s_{x_2} s_{x_1}$. We call this sequence the reflection sequence for the word $s_{x_1} \cdots s_{x_N}$. The length of $w$, written $\ell(w)$, is the length of the shortest expression for $w$ as a product of the simple generators and a product which achieves this minimal length is called reduced. If $s_{x_1} \cdots s_{x_N}$ is reduced then all the elements of the reflection sequence for $s_{x_1} \cdots s_{x_N}$ are distinct. Furthermore, in this case, $s_{x_1} \cdots s_{x_{i-1}} \alpha_{x_i}=\alpha_{s_{x_1} \cdots s_{x_{i-1}} s_{x_i} s_{x_{i-1}} \cdots s_{x_1}}$ (as opposed to $-\alpha_{s_{x_1} \cdots s_{x_{i-1}} s_{x_i} s_{x_{i-1}} \cdots s_{x_1}}$). If $s_{x_2} \cdots s_{x_N}$ is reduced then $s_{x_1} s_{x_2} \cdots s_{x_N}$ is reduced if and only if $s_{x_1}$ is \emph{not} an inversion of $s_{x_2} \cdots s_{x_N}$. 

The previous three paragraphs are very well known; a good reference for this material and far more concerning Coxeter groups is \cite{Humph}.  We now describe one additional combinatorial tool and one geometric tool. For $i$ between $1$ and $n$, define the map $\pi_i : W \to W$ by $\pi_i(w)=s_i w$ if $\ell(s_i w) > \ell(w)$ and $\pi_i(w)=w$ otherwise. This is sometimes known as the degenerate Hecke action. The condition that $\ell(s_i w) > \ell(w)$ is equivalent to the condition that $s_i$ is \emph{not} an inversion of $w$. Note that, if $s_{x_1} \cdots s_{x_N}$ is reduced then $\pi_{x_1} \cdots \pi_{x_N} e=s_{x_1} \cdots s_{x_N}$. Also, if $s_i$ and $s_j$ commute, so do $\pi_i$ and $\pi_j$.  We call $\pi_{x_1} \cdots \pi_{x_N} e$ the \emph{Demazure product} of $x_1 \cdots x_n$. For a quick introduction to the properties of the Demazure product, see Section 3 of \cite{KM}.

Let $c=s_{x_1} \ldots s_{x_n}$ be a Coxeter element of $W$. A simple reflection $s$ is called \emph{initial in $c$} if it is the first letter of some reduced word for $c$ and is called \emph{final in $c$} if it is the last letter of some reduced word for $c$.  So $s_{x_1}$ is initial in $c$ and $s_{x_n}$ is final in $c$. We define a skew symmetric bilinear form $\omega_c$ on $V$ by $\omega_c(\alpha_{x_i}, \alpha_{x_j})=B(\alpha_{x_i}, \alpha_{x_j})$ for $i<j$. (By skew-symmetry, $\omega_c(\alpha_{x_i}, \alpha_{x_j})=-B(\alpha_{x_i}, \alpha_{x_j})$ for $i>j$ and $\omega_c(\alpha_i,\alpha_i)=0$.) It is easy to check that $\omega_c$ does not depend on the choice of representation for $c$. We have
\begin{prop} \label{omega}
With the above notations, we have

\begin{enumerate}
\item For all $v$ and $w \in V$, we have $\omega_{s_{x_1} c s_{x_1}}(s_{x_1} v, s_{x_1} w)=\omega_c(v,w)$.
\item For all positive roots $\alpha_t$, $\omega_{c}(\alpha_{s_{x_1}}, \alpha_t) \leq 0$, with equality if and only $s_{x_1}$ and $t$ commute.
\item For all positive roots $\alpha_t$, $\omega_{c}(\alpha_{s_{x_n}}, \alpha_t) \geq 0$, with equality if and only $s_{x_n}$ and $t$ commute.
\end{enumerate}
\end{prop}

\begin{proof}
We first check property (1). Let $c=s_1 \cdots s_n$ with $s=s_1$. We recall the formula $sv=v-B(\alpha_s,v) \alpha_s$. It is enough to check the formula in the case that $v$ and $w$ are simple roots, say $v=\alpha_{s_i}$ and $w=\alpha_{s_j}$ with $i<j$. We consider two cases.

\textbf{Case 1: $i=1$.} Then 
\begin{multline*}
\quad \omega_{scs}(s \alpha_{s}, s \alpha_{s_j}) = \omega_{scs} (- \alpha_s, \alpha_{s_j} - B(\alpha_s, \alpha_{s_j}) \alpha_s) = \\ - \omega_{scs}(\alpha_s, \alpha_{s_j}) =  B(\alpha_s, \alpha_{s_j})  = \omega_c(\alpha_{s}, \alpha_{s_j}). \quad
\end{multline*}
We used that $s$ is final in $scs$  and initial in $c$ to deduce the signs in the last two equalities.

\textbf{Case 2: $i>1$.} Then
\begin{multline*}
\omega_{scs}(s \alpha_{s_i}, s \alpha_{s_j}) = \omega_{scs}( \alpha_{s_i} - B(\alpha_{s}, \alpha_{s_i}) \alpha_s, \alpha_{s_j} - B(\alpha_{s}, \alpha_{s_j}) \alpha_s) = \\ \omega_{scs}(\alpha_{s_i}, \alpha_{s_j}) - B(\alpha_s, \alpha_{s_i}) \omega_{scs}(\alpha_s, \alpha_{s_j}) - B(\alpha_s, \alpha_{s_j}) \omega_{scs}(\alpha_{s_i}, \alpha_{s})
\end{multline*}
Now, $s$ is final in $scs$, so $\omega_{scs}(\alpha_s, \alpha_{s_j})=-B(\alpha_s, \alpha_{s_j})$ and $\omega_{scs}(\alpha_{s_i}, \alpha_{s})=B(\alpha_{s_i}, \alpha_s)$. Thus, 
\begin{multline*}
- B(\alpha_s, \alpha_{s_i}) \omega_{scs}(\alpha_s, \alpha_{s_j}) - B(\alpha_s, \alpha_{s_j}) \omega_{scs}(\alpha_{s_i}, \alpha_{s})= \\  B(\alpha_s, \alpha_{s_i}) B(\alpha_s, \alpha_{s_j}) - B(\alpha_s, \alpha_{s_j}) B(\alpha_{s_i}, \alpha_{s})=0
\end{multline*}
and we deduce that
$$\omega_{scs}(s \alpha_{s_i}, s \alpha_{s_j})=\omega_{scs}(\alpha_{s_i}, \alpha_{s_j}) =B(\alpha_{s_i}, \alpha_{s_j}) = \omega_c(\alpha_{s_i}, \alpha_{s_j}).$$
We have used that $s_i$ comes before $s_j$ in a reduced word for $scs$, as well as in a reduced word for $c$. This concludes the proof of (1).

Because $s$ is initial in $c$, $\omega_{c}(\alpha_s, \alpha_t)= B(\alpha_s, \alpha_t)$. We have $B(\alpha_s, \alpha_t) \leq 0$, and strict inequality tautologically occurs unless $B(\alpha_s, \alpha_t)=0$. But $B(\alpha_{s}, \alpha_{t})=0$ if and only if $st=ts$. This proves property (2), and the proof of property (3) is very similar.
\end{proof}

\section{Admissible Sequences} \label{admissible}

This section essentially recapitulates (part of) section 2 of Kleiner and Pelley and we will try to repeat the terminology from Kleiner and Pelley as much as possible. Let $\Gamma$ be a finite graph and let $c$ be an acyclic orientation of $\Gamma$. If $x$ is a sink of $(\Gamma, c)$, we write $s_x c s_x$ for the orientation of $\Gamma$ obtained by reversing all edges coming into $x$. A sequence $x_1$, $x_2$, \dots, $x_N$ of vertices of $\Gamma$ is called admissible if $x_1$ is a sink of $(\Gamma, c)$, $x_2$ is a sink of $(\Gamma, s_{x_1} c s_{x_1})$, $x_3$ is a sink of $(\Gamma, s_{x_2} s_{x_1} c s_{x_1} s_{x_2})$ and so forth. We put an equivalence relation on the set of admissible sequences by setting two sequences to be equivalent if they differ only by interchanging the order of non-adjacent vertices. Let $\sss$ denote the set of admissible sequences modulo this equivalence relation. When it is necessary to emphasize the dependence on $c$, we will write $\sss_c$ and say that elements of $\sss_c$ are $c$-admissible. The following obvious observation will be of repeated importance:

\begin{prop}
 If $s$ and $t$ are two vertices of $\Gamma$, connected by an edge which is $c$-oriented from $s$ to $t$, then the occurences of $s$ and $t$ in any $c$-admissible sequence must alternate, with $s$ coming first.
\end{prop}

We can now state a more general result, which immediately implies Theorem~\ref{Result1}.
\begin{theorem} \label{Result2}
 Let $W$ be an infinite, irreducible Coxeter group and let $c$ be a Coxeter element. Let $x_1 x_2 \ldots x_N$ be any $c$-admissible sequence. Then $s_{x_1} \cdots s_{x_N}$ is reduced.
\end{theorem}

%

We need a small combinatorial lemma first. For $u=[x_1 \ldots x_N] \in \sss$, let $\phi(u)_x$ be the number of occurrences of $x$ in $x_1 x_2 \cdots x_N$.So $\phi(u)$ is an integer-valued function on the vertices of $\Gamma$. We put the structure of a poset on $\sss$ by setting $u \curlyleq v$ if one can choose representatives $u_1 \ldots u_M$ and $v_1 \ldots v_N$ for the equivalence classes $u$ and $v$ such that $M \leq N$ and $u_i=v_i$ for $i \leq M$. 
\begin{prop} \label{poset}
We have   $u_1 \ldots u_M \curlyleq v_1 \ldots v_N$ if and only if $\phi(u_1 \ldots u_M)_x \leq \phi(v_1 \ldots v_N)_x$ for every $x$ between $1$ and $n$.
\end{prop}

This is part of \cite[Proposition~3.2]{KP}; we porovide a short proof.

\begin{proof}
The ``only if'' direction is obvious, we prove the ``if'' direction by induction on $M$. The base case $M=0$ is obvious. Note that $u_1$ is necessarily a source of $\Gamma$. Since $\phi(u_1 \ldots u_M)_{u_1} \leq \phi(v_1 \ldots v_N)_{u_1}$, the vertex $u_1$ must occur somewhere in $v_1 \ldots v_N$; let $v_r$ be the first appearance of $u_1$. Let $w$ be any vertex neighboring $u_1$, we claim that $w$ does not occur among $v_1$, $v_2$, \dots, $v_{r-1}$. This is because, as noted above, the occurrences of $u_1$ and $w$ in $v_1 \ldots v_N$ must be interlaced, with $u_1$ appearing first. So  $v_1 \ldots v_N$ is equivalent to $v_r v_1 v_2 \ldots v_{r-1} v_{r+1} \ldots v_N$. By induction, $u_2 u_3 \ldots u_M \curlyleq v_1 v_2 \ldots v_{r-1} v_{r+1} \ldots v_N$ in $\sss_{s_{u_1} c s_{u_1}}$, so $u_1 \ldots u_M \curlyleq v_1 \ldots v_N$ in $\sss_{c}$.
\end{proof}

\begin{cor}
The map $\phi$ is injective.
\end{cor}

\begin{proof}
 If $\phi(u_1 \ldots u_M)=\phi(v_1 \ldots v_N)$ then $u_1 \ldots u_M \curlyleq v_1 \ldots v_N$ and $u_1 \ldots u_M \curlygeq v_1 \ldots v_N$ so $u_1 \ldots u_M$ is equivalent to $v_1 \ldots v_N$.
\end{proof}

\textbf{Remark:} Kleiner and Pelley characterize the image of $\phi$, and use it to show that the poset $\sss$ is a distributive semi-lattice. Hohlweg, Lange, and Thomas, in \cite{HLT}, study the lower interval of reduced words in $\sss$ (for $W$ a finite Coxeter group) and show that it is a distributive lattice as well. Hopefully, these lattices are related to the appearance of lattice theory in Nathan Reading's and my work. (See \cite{Reading1}, \cite{Reading2}, \cite{RS}.)

\section{The Crucial Lemmas}

Now, let $W$ be a Coxeter group and $\Gamma$ its Dynkin diagram. As discussed above, there is a bijection between Coxeter elements of $W$ and acyclic orientations of $\Gamma$, and we will feel free to use the same symbol to refer both to an orientation and the corresponding Coxeter element. In this section, we will establish the following.

\begin{prop} \label{MT}
 Let $x_1 \cdots x_N$ be of minimal length among all $c$-admissible sequences with Demazure product $w$. Then the word $s_{x_1} \cdots s_{x_N}$ is reduced and $w=s_{x_1} \cdots s_{x_N}$.
\end{prop}

%
%
%
%

Note that, at this point, we have not made any assumptions about $W$ being infinite or irreducible. That will come later, when we apply this result to prove that particular words are reduced. The key innovation of this note is contained in the following lemma, which will be essential in the proof of Proposition~\ref{MT}. 

\begin{lemma} \label{order}
Suppose that $x_1 \ldots x_N$ is $c$-admissible and $s_{x_1} \ldots s_{x_N}$ is a reduced word of $W$. Let $t_i$ be the reflection $s_{x_1} \cdots s_{x_{i-1}} s_{x_i} s_{x_{i-1}} \cdots s_{x_1}$. Then $\omega_c( \alpha_{t_i}, \alpha_{t_j}) \leq 0$ for $i < j$, and equality implies that $t_i$ and $t_j$ commute. 
\end{lemma}

\begin{proof}
Our proof is by induction on $i$. If $i=1$, then $x_1$ is a sink of $c$ and the result is part (ii) of Proposition~\ref{omega}. If $i>1$ then, by induction, we have $\omega_{s_{x_1} c s_{x_1}}(\alpha_{s_{x_1} t_i s_{x_1}}, \alpha_{s_{x_1} t_j s_{x_1}}) \leq 0$, with equality if and only if $B(\alpha_{s_{x_1} t_i s_{x_1}}, \alpha_{s_{x_1} t_j s_{x_1}})=0$. But, since $s_{x_1} \cdots s_{x_N}$ is reduced, we know that $\alpha_{s_{x_1} t_i s_{x_1}}=s_{x_1} \alpha_{t_i}$ so, by part (i) of Proposition~\ref{omega}, we have
$$ \omega_c( \alpha_{t_i}, \alpha_{t_j})= \omega_{s_{x_1} c s_{x_1}}(s_{x_1} \alpha_{t_i}, s_{x_1} \alpha_{t_j} ) \leq 0 $$
as desired. Moreover, since $t$ and $u$ commute if and only if $sts$ and $sus$ do, the equality conditions match.
\end{proof}

We now begin the proof of Proposition~\ref{MT}. Our proof is by induction on $N$; if $N=1$ the result is trivial. Let $w$ and $x_1 \cdots x_N$ be as in the statement of Proposition~\ref{MT} with $N>1$ and assume that the result is known for all $c$ and for all smaller values of $N$. Abbreviate $s=s_{x_1}$ and $w'=\pi_{x_2} \cdots \pi_{x_N} e$. We note that $x_2 \cdots x_N$ is of minimal length among $scs$-admissible sequences $y_1 \cdots y_M$ with Demazure product $w'$ -- if $y_1 \cdots y_M$ were a shorter such sequence then $x_1 y_1 \cdots y_M$ would be a shorter $c$-admissible sequence with Demazure product $w$. So, by induction, $s_{x_2} \cdots s_{x_N}$ is reduced and is equal to $w'$. The only way that $s_{x_1} s_{x_2} \cdots s_{x_N}$ might not be reduced then is if $s$ is an inversion of $w'$ and $w=w'$. We adopt the notation $u_i$ for $s_{x_2} \cdots s_{x_{i-1}} s_{x_i} s_{x_{i-1}} \cdots s_{x_2}$ where $2 \leq i \leq N$, so the $u_i$ are the inversions of $s_{x_2} \cdots s_{x_N}$. Suppose, for the sake of contradiction, that $s=u_a$ and, thus, $w=w'$.

Consider any $b$ between $a$ and $N$. On the one hand, by Proposition~\ref{order}, $\omega_{s_{x_1} c s_{x_1}}(\alpha_{u_a}, \alpha_{u_b}) \leq 0$. (Recall that $x_2 \cdots x_N$ is reduced.) On the other hand, $u_a=s_{x_1}$ and $s_{x_1}$ is the final letter in $s_{x_1} c s_{x_1}$, so $\omega_{s_{x_1} c s_{x_1}}(\alpha_{u_a}, \alpha_{u_b}) \geq 0$ by part (iii) of Proposition~\ref{omega}. We deduce that, for all $b$ with $a < b \leq N$, we have $\omega_{s_{x_1} c s_{x_1}}(\alpha_{u_a}, \alpha_{u_b}) = 0$ and, by Proposition~\ref{order}, $u_a u_b=u_b u_a$. Thus, we deduce that $u_a$ commutes with $u_b$ for all $b$ with $a < b \leq N$.

Write $v_i=s_{x_a} \cdots s_{x_{i-1}} s_{x_i} s_{x_{i-1}} \cdots s_{x_a}$ and $w=s_{x_2} \cdots s_{x_{a-1}}$, so $u_i=w v_i w^{-1}$. Then $s_{x_a}=v_a$ commutes with $v_b$ for all $b$ with $a < b \leq N$. From the identity $s_{x_b} = v_a v_{a+1} \cdots v_b \cdots v_{a+1} v_a$, we conclude that $s_{x_a}$ commutes with $s_{x_b}$ for all $b$ between $a$ and $N$. We will refer to this fact as ``the commuting property''.
 
 But now we are near a contradiction. Since $s_{x_a}$ commutes with $s_{x_b}$ for all $a < b \leq N$, we have
$$s_{x_2} \cdots s_{x_N}=s_{x_2} \cdots s_{x_{a-1}} s_{x_{a+1}} \cdots s_{x_N} s_{x_a},$$
and both products are reduced and equal to $w=w'$. Moreover, from the commuting property, $s=u_a$ is the last reflection in the reflection sequence for the reduced word $s_{x_2} \cdots s_{x_{a-1}} s_{x_{a+1}} \cdots s_{x_N} s_{x_a}$. So the word $s_{x_2} \cdots s_{x_{a-1}} s_{x_{a+1}} \cdots s_{x_N}$ is also reduced and $s$ does not occur at all in the reflection sequence for this word. We thus deduce that the word $s s_{x_2} \cdots s_{x_{a-1}} s_{x_{a+1}} \cdots s_{x_N}$ is reduced.

From the relation $s=u_a$, we know that
$$s s_{x_2} \cdots s_{x_{a-1}} s_{x_{a+1}} \cdots s_{x_N}=w$$
and the left hand side of this equation is reduced by the computations of the preceding paragraph. So the sequence $x_1 x_2 \cdots x_{a-1} x_{a+1} \cdots x_N$ has Demazure product $w$. But, from the commuting relation and the fact that $x_1 \cdots x_N$ is $c$-admissible, we know that $x_1 x_2 \cdots x_{a-1} x_{a+1} \cdots x_N x_a$ is $c$-admissible, and thus $x_1 x_2 \cdots x_{a-1} x_{a+1} \cdots x_N$ is certainly $c$-admissible. So $x_1 x_2 \cdots x_{a-1} x_{a+1} \cdots x_N$ is a $c$-admissible sequence with Demazure product $w$ that is shorter than $x_1 x_2 \cdots x_N$, contradicting our assumption of minimality. This contradiction concludes the proof of Proposition~\ref{MT}.

%

\section{Finishing the Proof}

Assume that $W$ is infinite and irreducible; recall that the second assumption simply means that the Dynkin diagram $\Gamma$ is connected. We now have a powerful tool to prove that certain words in $W$ are reduced. In this section, we fill apply this tool to prove that $s_{x_1} \cdots s_{x_N}$ is reduced for any $c$-admissible sequence $x_1 \cdots x_N$. Consider the sequence $w_k=(\pi_{1} \pi_{2} \cdots \pi_{n})^k$. Clearly, the sequence $\ell(w_k)$ is weakly increasing. We claim that in fact it is strictly increasing. If not, there is some $w=w_k=w_{k+1}$ with $\pi_{1} w = \pi_{2} w=\cdots = \pi_{n} w =w$. But then $s_i$ is an inversion of $w$ for every $i$ from $1$ to $n$. In an infinite Coxeter group, there is no element with this property. (In a finite Coxeter group, the only element with this property is the maximal element $w_0$.)

Therefore, $\ell(w_k) \geq k$. By Proposition~\ref{MT}, there is a $c$-admissible reduced word for each $w_k$, call this reduced word $\omega_k$; we know that $\omega_k$ has length at least $k$. Let $z$ be the letter that occurs most often in $\omega_k$, then $\phi(\omega)_z \geq k/n$. (Recall the map $\phi$ from Section~\ref{admissible}.) Now we use that $\Gamma$ is connected. Let $\delta$ be the diameter of the (unoriented) graph $\Gamma$. If $x$ and $y$ are adjacent vertices of $\Gamma$, then $x$ and $y$ alternate within $\omega_k$, so $|\phi(\omega_k)_x-\phi(\omega_k)_y| \leq 1$ and we deduce that $\phi(\omega_k)_x \geq k/n-\delta$ for any $x$. Let $M$ be the greatest number of times any letter occurs in $x_1 \cdots x_N$. Choosing $k$ large enough that $k/n-\delta \geq M$, we see that $\phi(\omega_k)_x \geq \phi(x_1 \cdots x_N)_x$ for any $x$ so, by Proposition~\ref{poset}, $x_1 \cdots x_N$ is equivalent to a prefix of the reduced word $\omega_k$. In particular, $x_1 \cdots x_N$ is reduced. This concludes the proof of Theorem~\ref{Result2} and hence proves Theorem~\ref{Result1}.

We note one variant of this argument. Suppose we try using the above argument in a finite Coxeter group. We must have $w_k=w_0$ for $k$ sufficiently large. So we can deduce from Proposition~\ref{MT} that there is a $c$-admissible sequence with product $w_0$. This result also occurs in \cite{HLT}. The authors of that paper characterize ``$c$-singletons'' as those elements of $W$ which have a reduced word which is a prefix of a $c$-admissible reduced word with product $w_0$. The argument of this note is the shortest proof I know of that $c$-admissible sequences giving reduced words for $w_0$ exist at all.

\section{Acknowledgments}

I am grateful to Nathan Reading for making me aware of the work of Kleiner and Pelley, as well as several other references cited here. Andrei Zelevinsky and Mark Kleiner both encouraged me to pursue a purely combinatorial proof. I was funded during this research by a research fellowship from the Clay Mathematics Institute.


\thebibliography{9}

\bibitem{FZ} S. Fomin and A. Zelevinsky, \emph{Cluster algebras IV}, Compositio Mathematica \textbf{143} (2007), 112-164

\bibitem{Howlett} R. B. Howlett, \emph{Coxeter groups and $M$-matrices}, Bulletin of the London Mathematical Society \textbf{14} (1982) no. 2 137--141

\bibitem{HLT}
C. Hohlweg, C. Lange and H. Thomas, \emph{Permutahedra and Generalized Associahedra}, \texttt{arXiv:0709.4241}

\bibitem{Humph}
J.~Humphreys,
\textit{Reflection Groups and Coxeter Groups.}
Cambridge Studies in Advanced Mathematics {\bf 29},
Cambridge Univ. Press, 1990.

\bibitem{KP} M. Kleiner and A. Pelley \emph{Admissible sequences, preprojective modules, and reduced words in the Weyl group of a quiver} \texttt{arXiv:math.RT/0607001}

\bibitem{KT} M. Kleiner and H. R. Tyler, \emph{Admissible sequences and the preprojective component of a quiver} Advances in Mathematics \textbf{192} (2005) no. 2 376--402

\bibitem{KM} A. Knutson and E. Miller, \emph{Subword complexes in Coxeter groups}, Advances in Mathematics \textbf{184} (2004), no. 1, 161--176. 

\bibitem{KMOTU} A. Kuniba, K. Misra, M. Okado, T. Takagi, J. Uchiyama
\emph{Crystals for Demazure modules of classical affine Lie algebras}
J. Algebra \textbf{208} (1998), no. 1, 185--215.

\bibitem{Reading1} N. Reading \emph{Cambrian Lattices} Adv. Math. \textbf{205} (2006), no. 2, 313-353.

\bibitem{Reading2} N. Reading \emph{Sortable elements and Cambrian lattices} Algebra Universalis \textbf{56} (2007) no. 3-4, 411-437

\bibitem{RS} N. Reading and D. Speyer \emph{Cambrian Fans}, JEMS to appear, \texttt{arXiv:math.CO/0606201}
\end{document}